\documentclass[12pt,oneside,a4paper]{amsart}
\usepackage{microtype}
\usepackage{amssymb}
\usepackage{enumerate}

\usepackage[T1]{fontenc}
\usepackage{newtxtext,newtxmath}
\mathcode`\;="303B

\usepackage{tikz}
\usetikzlibrary{decorations.pathreplacing}
\usepackage{todonotes}

\renewcommand{\epsilon}{\varepsilon}
\renewcommand{\phi}{\varphi}

\newcommand{\X}{\mathcal{X}}
\newcommand{\Y}{\mathcal{Y}}

\newtheorem{theorem}{Theorem}

\newtheorem{lemma}[theorem]{Lemma}
\newtheorem{definition}[theorem]{Definition}
\newtheorem{remark}[theorem]{Remark}

\makeatletter
\renewcommand\subsection{\@startsection{subsection}{2}%
  \z@{.5\linespacing\@plus.7\linespacing}{.1\linespacing}%
  {\normalfont\scshape}}
\makeatother

\title{$C^{1,1}$ regularity for principal-agent problems}
\author{Robert J. McCann
  \and
  Cale Rankin
  \and
Kelvin Shuangjian Zhang}
\address{Department of Mathematics, University of Toronto, Toronto Ontario M5S 2E4 Canada {\tt mccann@math.toronto.edu}
and {\tt cale.rankin@utoronto.ca}}
\address{Fields Institute for Research in the Mathematical Sciences, 222 College Street, Toronto Ontario  M5T 3J1 Canada {\tt crankin@fields.utoronto.ca}}
\address{ School of Mathematical Sciences,
Fudan University, Shanghai,
CHINA 200433
{\tt ksjzhang@fudan.edu.cn}}

\thanks{$^*$ Robert McCann's research is supported in part by the Canada Research Chairs program CRC-2020-00289 and Natural Sciences and Engineering Research Council of Canada Discovery Grants RGPIN--2020--04162. Cale Rankin is supported by the Fields Institute for Research in Mathematical Sciences. The authors are grateful to Luis Caffarelli for making his unpublished notes with Lions available to us,  and to Toronto's Fields Institute for the Mathematical Sciences, where much of this work was performed.    \copyright \today \\
MSC Classification Primary 49N60, 91B43; Secondary 35R35, 52A41, 58E50, 90B50.
}

\begin{document}

\begin{abstract}
  We prove the interior  $C^{1,1}$ regularity of  the indirect utilities which solve a subclass of principal-agent problems originally considered by Figalli, Kim,  and McCann. Our approach is based on construction of a suitable comparison function which, essentially, allows one to pinch the solution between parabolas.  The original ideas for this proof arise from an earlier, unpublished, result of Caffarelli and Lions for bilinear preferences which we extend here to general quasilinear benefit functions. We give a simple example which shows the  $C^{1,1}$ regularity is optimal.
\end{abstract}
\maketitle

Principal-agent problems are a class of economic models with applications to tax policy, regulation of public utilities, product line design and contract theory; for references see \cite{MR2888826,kstz2022,MR3957395}.
Mathematically these are optimization problems in the calculus of variations. A subclass of these problems which is economically relevant and mathematically tractable was studied by Figalli, Kim, and McCann \cite{MR2888826}. In this article we show the indirect utilities which solve the problems they considered are locally $C^{1,1}$; equivalently, the maps from agent types to products selected are Lipschitz continuous locally. Even in simple cases, where the minimizer may be computed explicitly, the minimizer is not $C^2$. Thus the $C^{1,1}$ regularity is optimal. In this introduction we outline the formulation of, and a bit of history on, the regularity of these problems. Then, in Section \ref{sec:sett-stat-results}, we provide precise details and a full formulation of the setting and our results. 

We are interested in minimizing a functional
\begin{equation}
  \label{eq:functional}
  L[u]:= \int_{\X}F(x,u,Du) \ dx,
\end{equation}
over the set of $b$-convex functions. Here, $b$-convexity is a generalized notion of convexity which arises in the contexts of incentive compatibility \cite{Rochet87} and independently, optimal transport \cite{Villani03}; in the bilinear case $b(x,y) = x \cdot y$ of Rochet and Chon\'e \cite{rochet-chone} it reduces to plain convexity (plus the option of a gradient constraint).  The quantity of economic interest, the principal's price menu, is a smooth function of the minimizer $u$ and its gradient $Du$ (further details in Section \ref{sec:backgr-mater-conv}). Thus regularity results for $u$ answer questions of economic interest such as ``do similar consumers  necessarily consume similar products?''.  
(By similar we mean suitably close in the spaces parameterizing consumer and product types respectively.)
 
The regularity of minimizers to \eqref{eq:functional} is complicated by the set of functions over which the applications  compel us to minimize. Indeed, even in the simplest case when the set of functions is the cone of convex functions we are prevented  from employing the usual techniques in the calculus of variations: Two sided interior perturbations of convex functions are, in general, not convex and, subsequently, it is difficult to derive the Euler--Lagrange equations; however, see Kolesnikov--Sandomirsky--Tsyvinski--Zimin \cite{kstz2022}, McCann--Zhang \cite{McCannZhangArxiv}, Rochet--Chon\'e \cite{rochet-chone} and their references for recent progress.

The framework in which Figalli, Kim, and McCann worked has links to optimal transport: the mapping between consumers and the product they purchase is the optimal transport map between the density of consumers and the density of products sold. Thus one might hope to bring the regularity theory from optimal transport (namely \cite{MR2188047,MR3067826,MR2512204,DePhilippisFigalli15}) to bear on the principal-agent problem. However, apriori we are studying regularity of the optimal transport map in a situation where the target data (density of products sold) is unknown. Thus, once again, the usual approach, via the Monge--Amp\`ere equation, does not apply in our setting. Numerical examples e.g. \cite{MR3474490}  \cite{NBERw30015} and theory \cite{armstrong1996} \cite{rochet-chone} \cite{McCannZhangArxiv} show the associated data for the Monge--Amp\`ere type equation is singular satisfying neither a lower bound with respect to the Lebesgue measure, nor absolute continuity with respect to the Lebesgue measure. 

Because the standard approaches to regularity are unavailable we rely on more bare-handed techniques. We use the only information available --- that $u$ is a minimizer to the functional \eqref{eq:functional}. Our proof relies on the construction of a comparison function. This technique, time immemorial in PDEs, allows one to turn the inequality $L[u] \leq L[v]$, which holds for all admissible functions $v$, into an inequality which directly implies the $C^{1,1}$ regularity. 

The literature on these techniques applied to the principal-agent problem is limited. In the special case of a bilinear benefit function the key results are a global $C^{1}$ result of {Rochet--Chon\'e \cite{rochet-chone}
and Carlier--Lachand-}Robert~\cite{MR1811127}, and an interior $C^{1,1}$ result that appears in some unpublished notes by Caffarelli and Lions \cite[circa 2006]{CaffarelliLions}. Our results and method of proof are based on the work of Caffarelli and Lions, but we improve on it also: we complete our analysis of the problem without separating it into the two cases they consider, one of which seems obscure to us. 

For general quasilinear benefit functions, those related to costs in optimal transport, {an interior analog of the $C^1$ result of Rochet--Chon\'e and Carlier--Lachand-Robert was obtained} by Chen~\cite{ChenThesis,MR4531938} (who made use of some of Caffarelli and Lions's techniques). In this paper we deal with the same problem that Chen did but prove the analogue of Caffarelli and Lions's $C^{1,1}$ result. We employ, as well as the techniques of the aforementioned works, some results and methods from the optimal transport literature, namely Figalli, Kim, and McCann \cite{MR2888826,MR3067826} and also Chen and Wang \cite{MR3419751}.

Our main result is Theorem \ref{thm:main} stated in Section \ref{sec:math-form} after introducing precisely the setting we're working in. In Section \ref{sec:simplest-case-an} we provide an exposition of 
Caffarelli and Lions's result for the bilinear setting. We include this section because here the geometric ideas are clearest, yet the result does not appear in the literature.  Following this we recall in Section \ref{sec:backgr-mater-conv} some background material, before proving our key results in Sections \ref{sec:key-lemma} and \ref{sec:proof-lemma}.

\section{Setting and statement of results}
\label{sec:sett-stat-results}

\subsection{Economic model}
\label{sec:economic-model}

Let us now heuristically explain the economic setting. The mathematical details are given in the next subsection. A monopolist is concerned with selling indivisible products (e.g., cars, houses, smart phones) to consumers. We consider a space of consumers $\X \subset \mathbf{R}^n$ and a space of products $\Y \subset \mathbf{R}^n$. We assume consumers are distributed according to some measure $\mu \ge 0$ on $\X$ and that product $y$ costs the monopolist price $c(y)$. The goal of the monopolist is to assign to each product $y \in \Y$ a price $v(y) \in \mathbf{R}$ in such a way that the monopolist's resulting profit is maximized. The function $v:\Y \rightarrow \mathbf{R}$ is called the price menu. Each consumer may purchase at most one product. The benefit a consumer $x$ gains from a product $y$ is given by  $b(x,y)$, and thus the utility obtained by consumer $x$ on purchasing product $y$ is $b(x,y) - v(y)$. A rational consumer will maximize their utility. That is, they will purchase the product $y$  which realizes
\begin{equation}
  \label{eq:u-econ-def}
  u(x) := \sup _{y \in \Y} b(x,y) - v(y).
\end{equation}
Employing the notation
\begin{equation}
  \label{eq:yu-economic-def}
  Yu(x) = \mathop{\rm argmax}\limits_{y \in \Y}b(x,y) - v(y)
\end{equation}
for the product purchased by $x$, the monopolist's profit is
\[L:= \int_{\X}v(Yu(x)) - c(Yu(x)) d \mu(x),\]
and their goal is to maximize this quantity over all choices of price menu $v$. To incentivize the monopolist not to price products too high we assume there is a given ``null'' product (or ``outside option'' ) $y_\emptyset$, which the monopolist is compelled to offer at a fixed price $v(y_\emptyset)$. This results in an additional requirement on the allowable price menus; the resulting utilities must satisfy $u(x) \geq b(x,y_\emptyset) - v(y_\emptyset)$ for all $x$. We now more precisely define our setting.

\subsection{Mathematical formulation}
\label{sec:math-form}

Figalli, Kim, and McCann were able to formulate the above problem using the framework of convexity in optimal transport. This framework strikes a balance between encompassing a broad class of benefit functions $b(x,y)$ whilst allowing the development of a convexity theory.

We take $\X,\Y \subset \mathbf{R}^n$ to be bounded domains. Denote their closures by $\overline{\X}$ and $\overline{\Y}$. We require $b$ to be smooth on $\overline{\X}\times \overline{\Y}$. For notation  we denote by $b_x(x_0,y)$ the differential of the mapping $x \in \X \mapsto b(x,y)$ evaluated at $x = x_0$, similarly for $b_y(x,y_0)$. We assume the following conditions on $b$, which we take from \cite{MR4531938}.  \\
\textbf{B1. } For each $(x_0,y_0) \in \overline{\X} \times \overline{\Y}$ the mappings
\begin{align*}
  &y \in \overline{\Y} \mapsto b_x(x_0,y)\\
  \text{ and }&x \in \overline{\X} \mapsto b_y(x,y_0)
\end{align*}
are diffeomorphisms onto their ranges.\\
\textbf{B2. }  For each $(x_0,y_0) \in \overline{\X} \times \overline{\Y}$ the sets $b_x(x_0,\Y)$ and $b_y(\X,y_0)$ are convex.

For our next condition let
\[Y:\{(x,b_{x}(x,y)) ; (x,y) \in  \overline \X \times \overline\Y\} \rightarrow \overline{\Y} \]
be defined by
\[ b_x(x,Y(x,p)) = p,\]
 and note $Y$ is well defined by B1. \\
\textbf{B3.} For $\xi,\eta \in \mathbf{R}^n$ there holds
\[ D_{p_kp_l}b_{x^ix^j}(x,Y(x,p))\xi^i\xi^j\eta_k\eta_l \geq0.\]
Here we follow the notation of \cite{MR2888826},  noting that B3 is a variant on the A3 and A3w conditions originally introduced for the regularity theory of optimal maps by Ma, Trudinger, Wang [10] and Trudinger, Wang [15]. The difference is A3 and A3w hold only in the case $\xi \cdot \eta = 0$. The  B3 condition is both more common in the economics literature and {more appropriate to} that setting: it's a necessary condition for the optimization problem to be convex \cite{MR2888826}. 

For such $b$ an associated convexity theory has been developed \cite{MR3067826,MR2188047}.

\begin{definition}
  A function $u:\X \rightarrow \mathbf{R}$ is called $b$-convex provided there is  $v:\Y \rightarrow \mathbf{R}$ such that 
  \begin{equation}
    \label{eq:b-conv-def}
    u(x) = \sup_{y \in \Y} b(x,y) - v(y).
\end{equation}
\end{definition}
We've already seen that the $b$-convexity of the consumers utility is a natural consequence of their rationality, as in \eqref{eq:u-econ-def}.
With this definition (note $u,v$ are finite) a $b$-convex function has a {\em support} at each point in its domain. That is, for each $x_0 \in \X$ there is $y_0 \in \overline{\Y}$ such that
\begin{equation}
  \label{eq:support-def}
  u(x) \geq u(x_0) + b(x,y_0) - b(x_0,y_0), \quad \text{ for all }x\in \X.
\end{equation}
By analogy with the bilinear case $b(x,y)=x\cdot y$, we call such lower bounds {\em $b$-affine} functions, which in this case forms a {\em $b$-support} for $u$ at $x_0$. For a given $x_0$ the set of all $y_0$ such that \eqref{eq:support-def} holds is denoted $Yu(x_0)$.  It's a singleton for almost every $x$. Defined in this way we see that $Yu(x)$ attains the supremum in \eqref{eq:b-conv-def} and therefore agrees with the definition \eqref{eq:yu-economic-def}. A $b$-convex function is called uniformly $b$-convex provided there is $\epsilon > 0$ such that $D^2u(x) - b_{xx}(x,Yu(x)) \geq \epsilon I$. 
If $b \not\in C^2$ we interpret these as directional derivative bounds in the distributional sense
(for which the ambiguity in $Yu$ on a set of measure zero is irrelevant).  The dual notion, $b^*$-convexity,  is obtained by switching the roles of $x$ and $y$. For example we define $b^*$-convexity as follows.
\begin{definition}
  A function $v:\Y \rightarrow \mathbf{R}$ is called $b^*$-convex provided there is  $u:\X \rightarrow \mathbf{R}$ such that 
  \begin{equation}
    \label{eq:bst-conv-def}
   v(y) = \sup_{x \in \X} b(x,y) - u(x).
\end{equation}
\end{definition}

By the equality $u(x) = b(x,Yu(x)) - v(Yu(x))$ we see that the principal-agent problem may be reformulated as follows. 

\textbf{Principal-agent problem} Minimize the functional
\begin{equation}
  L[u] := \int_{\X} [c(Yu(x)) - b(x,Yu(x)) + u(x)] \ d\mu(x) \label{eq:to-minimize}
\end{equation}
over the set
\[U_0 := \{u : \X \rightarrow \mathbf{R}; u \text{ is $b$-convex}, u(x) \geq a_\emptyset+b(x,y_\emptyset)\}\]
for a given $a_\emptyset,y_\emptyset$ which arise from the assumption of a ``null'' product.

If $u$ is the minimizer then the principal's optimal price menu (or its $b^*$-convex hull) is recovered as $v(y) = \sup_
{x \in \X} b(x,y) - u(x)$.

In this setting we can now state our main theorem.   Let $C^{0,1}(\X)$ denote the set of bounded Lipschitz real-valued functions on $\X$,
and $C^{1,1}_{loc}(X)$ the set of continuously differentiable functions whose values and derivatives are Lipschitz on each compact subset of $\X$
(i.e. on each $\X' \subset\subset X$).

\begin{theorem}\label{thm:main}
  Assume $b$ satisfies B1,B2,B3. Assume that $c$ is uniformly $b^*$-convex and $\mu = f \ dx$ where $f \in C^{0,1}(\X)$ satisfies $0 < \lambda \leq f(x)$. Then the solution of the principal-agent problem, $u$, satisfies $u \in C^{1,1}_{\text{loc}}(\X)$ .
\end{theorem}

\begin{remark}\label{rem:unif-conv}
 Although existence of solutions holds more generally \cite{Carlier01},  the condition that $c$ is $b^*$-convex is standard for the uniqueness theory \cite{MR2888826}. 
  Moreover a strengthening of convexity to uniform convexity is standard when one moves from uniqueness to regularity (consider for example degenerate ellipticity vs uniform ellipticity arising from convexity properties of a Lagrangian).  What's essential is that we obtain the conclusion
    \begin{equation}
 c(Y(x,p)) - b(x,Y(x,p))\label{eq:unif-conv}
\end{equation}
is uniformly convex in $p$, independent of $x$ (the dual result to Lemma \ref{lem:transformations} (i)). 
\end{remark}
 \begin{remark}
  It's straightforward to see Theorem \ref{thm:main} is optimal. Indeed, even in the simplest case: $\Omega = [0,1]\subset \mathbf{R}$, $b(x,y) = x y$, $c(y) = |y|^2/2$, and $f(x) = 1$, the minimizer is $C^{1,1}$ but not $C^2$. Indeed, we must minimize
  \[ L[u] = \int_{0}^1 {[ \frac{1}{2}u'(x)^2 + u(x) - xu'(x)]} \, dx.\]
  The Euler--Lagrange equation in the region of strict convexity is $u''(t) = 2$. Thus using that $u \in C^{1,1}_{\text{loc}}$ along with the Neumann condition $u'(1) = 1$ we may compute that the minimizer is the following function with discontinuous second derivative:
  \begin{align*}
    u(x) =
    \begin{cases}
      x^2 - x + 1/4 &\text{for }1/2 \leq x,\\
      0 &\text{for }x < 1/2.
    \end{cases}
  \end{align*}
A rigorous justification that the above $u$ is the minimizer follows, {e.g.,} from the duality result of the first and third authors \cite{McCannZhangArxiv}. 
  Similar explicit examples may be computed in two dimensions for $\Omega =
  B_1(0)$, $b(x,y) = x \cdot y$, $c(y) = |y|^2/2$ and $f(x) = 1$. Again, the minimizer is not $C^2$ across the boundary of the nonparticipation region. 
\end{remark}

 \section{The simplest case; an exposition of Caffarelli and Lions's result}
\label{sec:simplest-case-an}

In this section we prove Theorem \ref{thm:main} in the simplest setting: $b(x,y) = x \cdot y$, $f \equiv 1$ and $c(y) = |y|^2/2$. In this setting the result was proved by Caffarelli and Lions though never published. We present their result to introduce the reader to the main geometric ideas in our proof and indicate {both where we have simplified their argument, and where} the difficulties lie in extending the result to benefit functions satisfying B3.

\begin{theorem}
  Let $u$ minimize
  \[ \Phi[u]:= \int_{\X}\frac{1}{2}|Du|^2 + u - x \cdot Du \, dx,\]
  over the set of nonnegative convex functions. Then $u \in C^{1,1}_{\text{loc}}(\X)$. 
\end{theorem}
\begin{proof}
  The work of Carlier and Lachand-Robert \cite{MR1811127} implies $u \in C^{1}(\overline{\X})$.  We fix $x_0 \in \X$, and let $l_{x_0}(x) = u(x_0) + Du(x_0)\cdot(x-x_0)$ be the support plane at $x_0$. To show $u \in C^{1,1}_{\text{loc}}(\X)$ it suffices to show there is $C>0$ such
  \begin{equation}
    \label{eq:c-goal1}
    \sup_{B_r(x_0)}|u-l_{x_0}| \leq Cr^2,
  \end{equation}

  After translating and subtracting the support we assume $x_0 = 0$ and $l_{x_0} \equiv 0$. Thus, we aim to show
  \begin{equation}
    \label{eq:c-goal2}
     h:=\sup_{B_r(0)}u \leq Cr^2.
  \end{equation}
 Without loss of generality this supremum occurs at $re_1$, that is $h = u(re_1).$ 
Consider the perturbation
\begin{align*}
  u_{h}(x) &:= \text{max}\{u(x), \frac{h}{2r}(x^1+r)\}\\
  &= \text{max}\{u(x), \sigma_h(x)\},\text{ where } \sigma_h(x) :=\frac{h}{2r}(x^1+r).
\end{align*}
It's clear $u_h$ is admissible for the minimization problem. Namely $u_h \geq u$ and is convex. The key idea of the proof is that the minimality condition  $\Phi[u] \leq \Phi[u_h]$ can be reworked, via suitable estimates, into \eqref{eq:c-goal2}.

First we prove the essential ``slab containment'' estimate
\begin{equation}
  \label{eq:c-contain0}
 S := \{x \in \X; u_h(x) > u(x)\}  \subset \{x \in \X ; \vert x^1\vert \leq r\}.
  \end{equation}
  To obtain \eqref{eq:c-contain0} note if $x^1 < -r$ then $\sigma_h(x) < 0 = l_{x_0}(x) \leq u(x)$ so
\begin{equation}
  \label{eq:c-contain1}
  S \subset \{ x \in \X ; x^1 \geq - r\}.
\end{equation}
Containment in the other direction, follows by computing $D\sigma_h(re_1) - Du(re_1) = -\kappa e_1$ for some $\kappa \geq 0$. Thus the convex set $S = \{\sigma_h - u > 0\}$ has $\{x ; x^1 = r\}$ as a support hyperplane and \eqref{eq:c-contain0} follows. 

It's evident that convexity of $u$  was the essential property which allowed us to prove \eqref{eq:c-contain0}.  The extension to general benefit functions requires the condition B3 and delicate estimates in sufficiently small neighbourhoods to obtain an analogue of \eqref{eq:c-contain0}.

Now we employ $\Phi[u] \leq \Phi[u_h]$. We make the simplifying assumption $S \subset \subset \X$. When this is not true additional steps are needed. Caffarelli and Lions proceeded by an additional {perturbation argument whereas we employ a more straightforward dilation estimate.} In any case, when $S \subset \subset \X$ the inequality $\Phi[u] \leq \Phi[u_h] $ is
\[ \int_{\X} {[\frac{1}{2}|Du|^2 + u - x \cdot Du]} \, dx \leq \int_{\X} {[\frac{1}{2}|Du_h|^2 + u_h -x \cdot Du_h]} \, dx,\]
which, after employing $u_h = u$ on $\X \setminus S$ and integrating by parts becomes
\begin{equation}
  \label{eq:c-comb1}
  \frac{1}{2}\int_{S}|Du - D\sigma_h|^2 \, dx \leq (n+1) \int_{S} (\sigma_h-u) \, dx \leq C(n)|S|h.
\end{equation}
Here we've used $\int_{S} {(|Du|^2 - |D\sigma_h|^2)} = \int_{S}|Du-D\sigma_h|^2$. We claim
\begin{align}
\label{eq:c-comb3}  \text{ and } &\int_{S}|Du - D\sigma_h|^2 \, dx \geq C(n)\frac{h^2}{r^2}|S|,
\end{align}
which with \eqref{eq:c-comb1} yields \eqref{eq:c-goal2} and completes the proof. As we will show, in the bilinear case \eqref{eq:c-comb3} follows from careful estimates using the convexity of $u$. For more general benefit functions we rely on similar intuition but the estimates are more technical. The condition B3 is used to ensure on sufficiently small neighbourhoods similar, but more complicated, estimates hold. 

We conclude by indicating the proof of \eqref{eq:c-comb3}. Let $S/2$ be the dilation of $S$ by a factor $1/2$ with respect to the origin. Let $P(x)$ be the projection onto $\{x ;x^1 = 0\}$ defined by $P(x^1,x') = (0,x')$. For each $(0,x') \in P(S/2)$ the set 
\[ (P^{-1}((0,x')) \cap S) \setminus (S/2)\]
is two line segments of length less than $2r$ (by \eqref{eq:c-contain0}). Let $l_{x'}$ denote the one with greater $x^1$ coordinate and set $l_{x'} = [a_{x'},b_{x'}]\times \{x'\}$. Because $(a_{x'},x') \in \partial(S/2)$ there holds $(u - \sigma_h)(2a_{x'},2x') = 0$. In addition $(u-\sigma_h)(0) = -h/2$ so by convexity of $u$, $(u-\sigma_h)(a_{x'},x') \leq -h/4$. Finally, because $(b_{x'},x) \in \partial S$, $(u-\sigma_h)(b_{x'},x') = 0$. Thus Jensen's inequality implies
\[ \int_{l_{x'}}|Du - D\sigma_h|^2 \, d \mathcal{H}^1 \geq \frac{1}{2r}\left(\int_{a_{x'}}^{b_{x'}}D_{1}(u - \sigma_h)(t,x') \, dt\right)^2\geq  C\frac{h^2}{r}.\]
Using Fubini's theorem to integrate $S/2$ over all $l_{x'}$  we have
\begin{align*}
 \int_{S}|Du - D\sigma_h|^2\, dx &\geq\int_{P(S/2)}\int_{l_{x'}}|Du_h - D\sigma_h|^2 \, d\mathcal{H}^1 dx'\\
                                     &\geq C \mathcal{H}^{n-1}(P(S/2))\frac{h^2}{r}\geq C(n) |S|\frac{h^2}{r^2}
\end{align*}
Here the final inequality which is \eqref{eq:c-comb3}, and in which $\mathcal{H}^{n-1}$ denotes $n-1$ dimensional Hausdorff measure, follows because \eqref{eq:c-contain0} implies $\mathcal{H}^{n-1}(P(S/2))r \geq C(n) |S|$. 
\end{proof}

\section{Background material: convexification and localization}
\label{sec:backgr-mater-conv}

In this section we recall some standard tools for studying $b$-convexity \cite{MR3419751,MR3067826}. Fix a $b$-convex function $u:\X \rightarrow \mathbf{R}$ and $(x_0,y_0) \in \X \times \Y$. We assume $b_{x^i,y^j}(x_0,y_0) = \delta_{ij}$ which may always be realized after a suitable affine transformation. Define new coordinates
\begin{align}
 \tilde{x}(x) := b_y(x,y_0) - b_y(x_0,y_0),\label{eq:tilde-x-def}\\
  \tilde{y}(y) := b_x(x_0,y) - b_x(x_0,y_0).\label{eq:tilde-y-def}
\end{align}
We use the notation $x(\tilde{x})$ to denote the inverse, that is the corresponding $x$ such that for a given $\tilde{x}$ \eqref{eq:tilde-x-def} holds. Similarly for $y(\tilde{y})$, for example $x(0) = x_0$ and $y(0) = y_0$. We note all this is well--defined by condition B1. 

We introduce also the transformed functions
\begin{align}
  \tilde{u}(\tilde{x}) = u(x) - [u(x_0) + b(x,y_0) - b(x_0,y_0) ],\label{eq:tilde-u-def}\\
  \tilde{b}(\tilde{x},\tilde{y}) = b(x,y) - [b(x_0,y) + b(x,y_0) - b(x_0,y_0)],\label{eq:tilde-b-def}
\end{align}
for $x = x(\tilde{x})$ and $y = y(\tilde{y})$. The definitions given for $b$ hold for $\tilde{b}$. That is, there is a notion of $\tilde{b}$-convex functions and $\tilde u$ defined above is $\tilde b$-convex. There is a corresponding $\tilde{Y}\tilde{u}$ mapping and it is straightforward to verify that (with the obvious notation) $\tilde{Y}\tilde{u}(\tilde{x}) = \tilde{y}(Yu(x(\tilde{x})))$.

The following results, using the notation introduced above, render the abstract theory of $\tilde{b}$-convexity more tractable. 
\begin{lemma} {[Facts about $\tilde{b}$-convex geometry]}\label{lem:transformations}
  \begin{enumerate}[(i)]
  \item The function $\tilde{u}:\tilde{\X} \longrightarrow [0,\infty)$ is convex and $\tilde{u}(0) = 0$. It is uniformly convex if $u$ is uniformly $b$-convex.
  \item As a consequence, for all $h \geq 0$ the set
    \[ S_h := \{\tilde{x} \in \tilde{\X} ; \tilde{u}(\tilde{x}) < h \},\]
    is convex.
  \item We have the expansion
    \begin{equation}
      \label{eq:expansion}
       \tilde{b}(\tilde{x},\tilde{y}) = \tilde{x} \cdot \tilde{y} + a_{ij,kl}\tilde{x}^i\tilde{x}^j\tilde{y}^k\tilde{y}^l.
    \end{equation}
    Here the $a_{ij,kl}$ are smooth functions on $\overline{\X} \times \overline{\Y}$ which arise from taking a Taylor series to fourth order. 
  \end{enumerate}
\end{lemma}
  These results can be found in, or at least adapted from,  \cite{MR3419751,MR3067826,trudiner-wang-preprint}. For completeness we include a proof adapted to our setting in Appendix \ref{sec:proof-lemma-reflem}. 

A crucial result in optimal transport is the following estimate due to Loeper \cite{MR2506751}, and now commonly called the Loeper maximum principle.

\begin{theorem}(Loeper maximum principle \cite[Theorem 3.2]{MR2506751})\label{thm:loeper-max} 
  Let $x_0 \in \overline{\X}$ be given. Assume that $(y_t)_{t \in [0,1]}$ is a curve in $\Y$ such that
  \begin{equation}
    \label{eq:b*-seg-def}
    b_x(x_0,y_t) = (1-t)b_x(x_0,y_0)+tb_x(x_0,y_1).
\end{equation}
  Then for all $x \in \overline{\X}$ and $t \in [0,1]$ there holds
  \[ b(x,y_t) - b(x_0,y_t) \leq \max\{ b(x,y_0) - b(x_0,y_0) , b(x,y_1) - b(x_0,y_1) \}.\]
\end{theorem}

Loeper's result was proved under A3w, however under the current hypotheses (B3 as opposed to A3w), the function $t \in [0,1] \mapsto b(x,y_t) - b(x_0,y_t)$ is in fact convex \cite{MR2654086}.
We remark that when \eqref{eq:b*-seg-def} holds $(y_t)_{t \in [0,1]}$ is called the $b^*$-segment joining $y_0$ to $y_1$ with respect to $x_0$.

\section{Proof of Theorem \ref{thm:main} assuming a key lemma}
\label{sec:key-lemma}

In this section we state Lemma \ref{lem:main} which, whilst not of independent interest, simplifies the proof of Theorem \ref{thm:main}. Indeed, following the statement of Lemma \ref{lem:main} we immediately prove Theorem \ref{thm:main} and later, in Section \ref{sec:proof-lemma}, prove Lemma \ref{lem:main}.

\begin{lemma}[Geometry of a carefully chosen trial function]\label{lem:main}
  Let $X' \subset \subset \X$ and $d = \text{dist}(X',\partial \X)$. There is $r_0>0$ depending only on $b$ and $d$ and constants $C_1,C_2 >0$ depending on $b,c,f,d$ such that the following property holds: If $u : \X \rightarrow \mathbf{R}$ is $b$-convex and $x_0 \in \X', y_0 \in Yu(x_0)$, then provided $r < r_0$ and
  \[h:=\sup_{B_r(x_0)}[u(x) - (b(x,y_0)-b(x_0,y_0)+u(x_0))], \]
  is positive, there is a $b$-affine function $p_{y}(\cdot):=b(\cdot,y)+a$ such that:
  \begin{enumerate}
  \item The section $S:= \{x \in \X; u(x) < p_y(x)\}$ has positive measure. \\
  \item On $S$ we have
    \begin{equation}
\label{eq:height-comparison}
      \sup_{x \in S} [p_{y}(x) - u(x)] \leq h .
    \end{equation}
  \item There holds
    \begin{align}
  \label{eq:energy-comparison}     \frac{1}{|S|} \int_{S}  \big[\big(c(y) &- b(x,y)\big)  - \big(c(Yu(x)) - b(x,Yu(x)\big)\big]\ f(x) \  dx\\
                                         \nonumber          &\quad\leq  C_1h - C_2\frac{h^2}{r^2}. 
                                         \end{align}
  \end{enumerate}
\end{lemma}


\begin{remark}
  The constant $C_1$ depends on $\Vert f \Vert_{ C^{0,1}}$, $\Vert c \Vert_{C^1}$, $d$, $\|b\|_{C^{3,1}}$, $\|b_{x,y}^{-1}\|_{L^\infty}$, and diam$(\X\times\Y)$. The constant $C_2$ depends on all the quantities $C_1$ does and, in addition, $\lambda$ and the uniform $b^*$-convexity of $c$.  In our setting the local Lipschitz constant of any $b$-convex function is controlled by $\Vert b \Vert_{C^1}$.
\end{remark}
 
 Using Lemma \ref{lem:main}, Theorem \ref{thm:main} has the following short proof.

\begin{proof}[Proof. (Theorem \ref{thm:main})]
  Fix $\X' \subset \subset \X$, $x_0\in \X'$, and $y_0 \in Yu(x_0)$. Let
  \begin{align*}
  p_0(\cdot):=b(\cdot,y_0) - b(x_0,y_0)+u(x_0)
  \end{align*}
  be a $b$-support at $x_0$. Since the cost function is smooth, it's standard \cite[Proposition 1.2]{MR1351007} that to prove $u$ is $C^{1,1}_{\text{loc}}$ it suffices to prove that for all $r$ less than a given $r_0$ (independent of $u$, though dependent on $\X'$) there holds
  \[ \sup_{B_r(x_0)}|u - p_0| \leq Cr^2.\] Of course $u-p_0 \geq0$; we must prove $u - p_0 \leq Cr^2$. Here's where the lemma enters. Put
  \[ h = \sup_{B_r(x_0)} u - p_0,\]
  and note if $h = 0$, we're done. Thus, we assume $h>0$ and take the $b$-affine function $p_y$ and associated section $S$ given by Lemma \ref{lem:main}. We set
  \[ u_h = \text{max}\{u,p_y\}.\] Note $u_h$ is admissible for the Monopolist's problem: it is $b$- convex and not less than $u$. Thus, since $u$ is a minimizer for the principal-agent problem
  \[ 0 \leq L[u_h] - L[u]. \] Because $u_h$ differs from $u$ only on
  $S$ we compute $L[u_h]-L[u]$ and obtain
  \begin{align*}
    0 \leq \int_{S}\left[\big(c(y) - b(x,y)+p_y\big) - \big(c(Yu(x)) - b(x,Yu(x)) + u \big)\right]f(x) \ dx.
  \end{align*}
  From \eqref{eq:height-comparison} and \eqref{eq:energy-comparison} it's immediate that
  this inequality becomes
  \[ 0 \leq (C_1+\Vert f \Vert_{L^\infty})h - C_2\frac{h^2}{r^2}.\] This gives $h \leq Cr^2$ to establish the theorem.
\end{proof}

 \section{Proof of Lemma \ref{lem:main}: Section containment and energy estimate}\label{sec:proof-lemma}

 In the remainder of this paper we prove Lemma \ref{lem:main}. 
  A key element of the proof is to show that, in the coordinates described above, 
 the section $S$ can be chosen to lie in a slab of thickness proportional to $r$.  It is also crucial that the non-degenerate changes of coordinates exploited range over a compact family, so that the geometry cannot become too distorted.
 
\begin{proof}[Proof. (Lemma \ref{lem:main})]
  
  \textit{Step 1. (Setup)}  Chen showed $u\in C^1(\X)$ \cite{ChenThesis,MR4531938}.  Alternately, to make our proof logically independent of his,  one can use an approximation procedure, standard in the regularity theory for optimal transport, to assume $u$ is $C^1.$ Indeed for possibly non-differentiable $u$, \cite[Theorem 3.1]{MR2506751} yields a sequence of $C^{1}$ $b$-convex functions $u_k$ which converge to $u$ in the topology of local uniform convergence. Proving the Lemma for the $u_k$, then taking a limit yields the result for $u$. We note the limiting procedure uses that $C_1,C_2$ may be taken independent of $k$, the dominated convergence theorem, and the assumed estimates for $|c|,|b|,|f|$. 

  Working under the assumptions of Lemma \ref{lem:main} we fix $\X' \subset \subset \X$ and $x_0 \in \X'$ where without loss of generality $x_0 = 0$. Set $y_0 = Yu(x_0)$. Thus we can perform the change of variables and transformations \eqref{eq:tilde-x-def}--\eqref{eq:tilde-b-def}, after which the support at $0$ is $0$. By direct calculations we see if Lemma \ref{lem:main} holds for transformed quantities it holds for the original functions and coordinates, though with different constants $C_1,C_2.$ In particular distances in the $x$ and $\tilde{x}$ coordinates are comparable up to a constant depending on the eigenvalues of $b_{x,y}$ which are positive and bounded on $\overline{\X}\times \overline{\Y}$. 

  Thus it suffices to prove the lemma after applying the transformations \eqref{eq:tilde-x-def}--\eqref{eq:tilde-b-def}.  For ease of notation we retain $x,y,u,b$, that is, we don't switch to the notation $\tilde{x},\tilde{y},\tilde{u},\tilde{b}$ --- though keep in mind now $b$ satisfies \eqref{eq:expansion}.

  Now, without loss of generality $h:=\sup_{B_r}u$ occurs at $re_1$. Here $r < r_0$ where $r_0$ depending only on $b$ and $ d:= \text{dist}(\X',\partial \X)$ will be specified in the proof. Initially take $r_0 < d$. 

  Because $u$ attains a maximum over $\partial B_r$ at $re_1$ it has $0$ tangential derivative, that is $Du(re_1) = \kappa e_1$ for some $\kappa$. The convexity of (the transformed function) $u$ implies $\kappa \geq h/r$. Here $\kappa \leq \Vert b \Vert_{C^1}$. Since the gradient of the $b$-support at $re_1$ agrees with the gradient of $u$, $y_1 := Yu(re_1)$ satisfies $b_x(re_1,y_1) = \kappa e_1$. Moreover the transformed $b$ satisfies $b_x(re_1,0) = 0$. Therefore, by considering the $b^*$-segment joining $y_1$ to $0$ with respect to $re_1$, we obtain $y_{*}$ satisfying
  \begin{equation}
    \label{eq:y1/2-def}
    b_x(re_1,y_{*}) = \frac{h}{2r}e_1.
 \end{equation}
  The desired $b$-affine function will be  $p_{y_{*}}(x):=b(x,y_{*})-b(re_1,y_{*})+u(re_1)$. Because $p_{y_{*}}(0) - u(0) > 0$ (see \eqref{eq:origin-est}) we obtain that $S:=\{x ; u(x) < p_{y_{*}}(x)\}$ has positive measure.

\textit{Step 2. (Proof of \eqref{eq:height-comparison})}  Since $p_{y_1}(x):=b(\cdot,y_1) - b(re_1,y_1)+u(re_1)$ is a support at $re_1$ there holds $b(\cdot,y_1) - b(re_1,y_1) = p_{y_1} - h\leq u$.  The Loeper maximum principle (Theorem \ref{thm:loeper-max}) implies for all $x \in \X$
  \begin{equation}
    \label{eq:loeper-mp-app}
    b(x,y_{*})-b(re_1,y_{*}) \leq \text{max}\{0,b(x,y_1) - b(re_1,y_1)\} \leq u(x).
\end{equation}
  Here we've applied the Loeper maximum principle noting $0 = b(\cdot,0) - b(re_1,0)$ is also a support after the transformation \eqref{eq:tilde-b-def}. Estimate \eqref{eq:height-comparison} is obtained from \eqref{eq:loeper-mp-app} by adding $h = u(re_1)$ and subtracting $u(x)$. 

\textit{Step 3. (Basic estimates on $y_{*}$)}
  The rest of the proof, in which we obtain \eqref{eq:energy-comparison}, is based on estimates employing the Taylor expansion \eqref{eq:expansion}. As in the bilinear case, key to our result will be establishing the slab-containment condition
  \begin{equation}
    \label{eq:section-containment}
    S = \{x ; u(x) < p_{y_{*}}(x)\} \subset \{ x ; |x^1| \leq Cr\}
  \end{equation}
 for $C$ depending only on $b$.

  We begin with basic estimates showing $y_{*}$ is a small perturbation of $\frac{h}{2r}e_1$. By \eqref{eq:y1/2-def} and the expansion \eqref{eq:expansion} we have
  \begin{equation}
    \label{eq:y-expansion}
    \frac{h}{2r}e_1 = b_x(re_1,y_{*}) = y_{*} + f_{ijk}(re_1)^{i}(y_{*})^j(y_{*})^k,
 \end{equation}
  for $f_{ijk}$ a suitable smooth function arising from taking derivatives of \eqref{eq:expansion}. Choosing $r$ sufficiently small (depending only on $b$) and setting $\epsilon = r C(n)\sup|f_{ijk}|$ for $C(n)$ a suitable constant depending only on $n$ we have
  \begin{equation}
    \label{eq:angle-argument}
    \left\vert \frac{h}{2r}e_1 - y_{*}\right\vert \leq \epsilon|y_{*}|^2.
 \end{equation}
 An inequality of this form implies\footnote{We take these types of arguments from \cite{MR3419751}.} if $\theta$ is the angle $y_{*}$ makes with the $e_1$ axis, then $|\sin \theta| \leq \epsilon|y_{*}|$ (see Figure \ref{fig:geometric-arguments}(a)).

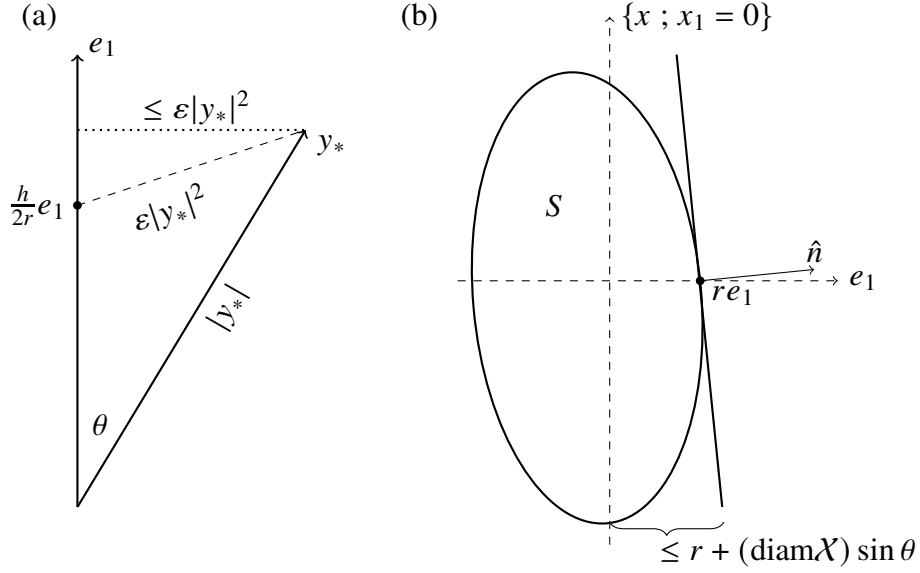
\begin{figure}[h]
  \centering
  \begin{tikzpicture}
    \node at (1.5,7.5) {(a)};
    \draw [->,thick] (2,1) -- (2,7);
    \node [right] at (2,7.1) {$e_1$};
    \draw [fill] (2,5) circle [radius=0.05];
    \node [left] at (2,5) {$\frac{h}{2r}e_1$};
    \draw [->,thick] (2,1) -- (5,6);
    \node [right] at (5,5.8) {$y_{*}$};
    \draw [dashed] (2,5) -- (5,6);
    \node [below,rotate=18.4] at (3.1,5.3) {$\epsilon|y_{*}|^2$};
    \draw [thick,dotted] (2,6) -- (5,6);
    \node [above,right] at (2.7,6.3) {$\leq \epsilon|y_{*}|^2$};
    \node [right,rotate=60] at (3.75,3.25) {$|y_{*}|$};
    \node [above] at (2.3,1.8) {$\theta$};
   
    \node at (6.5,7.5) {(b)};
    \draw [thick,rotate=5] (9,3) ellipse (1.5cm and 3cm); 
    \draw [fill] (10.19,4) circle [radius=0.05];
    \node [right] at (10.2,3.8) {$re_1$};
    \draw [dashed,->] (7,4) -- (12,4); 
    \node [right] at (12,4) {$e_1$};
    \draw [dashed,->] (9,0.5) -- (9,7.5);
    \node [right] at (9,7.5) {$\{x;x_1 = 0\}$};
    \draw [thick] (9.88,7) -- (10.48,1);
    \draw [->] (10.19,4)-- (11.69,4.15);
    \node [above] at (11.7, 4.1) {$\hat n$};
    \draw [decorate,decoration={brace,amplitude=5pt,mirror}] (9,0.8) -- (10.5,0.8);
    \node [below,right] at (9.5,0.4) {$ \leq r+(\text{diam} \X)\sin\theta $};
    \node [right] at (8,5) {$S$};
  \end{tikzpicture}
  \caption{Two geometric arguments. In (a) we illustrate how the estimate $ \left\lvert \frac{h}{2r}e_1 - y_{*} \right\rvert \leq \epsilon |y_{*}|^2$ yields the esimate $\sin \theta \leq \epsilon |y_{*}|$ (recall trigonometry). In (b) we illustrate that if the normal, $\hat n$, to a support plane to the convex body $S$ makes angle $\sin \theta \leq Cr$ with the $e_1$ axis $S \subset \{x ; x^1 \leq r + (\sin\theta)\text{diam}(\X) \leq Kr\}$.}
  \label{fig:geometric-arguments}
\end{figure}

  We perform a further change of coordinates, this time to the coordinates in which sections with respect to $p_{y_{*}}$ are convex. That is, set
  \begin{equation}
    \label{eq:y1/2-transform}
    \tilde{x}(x) := b_y(x,y_{*}),
  \end{equation}
  and consider  $\overline{u}, \overline{p}$ defined by $\overline{u}(\tilde{x}):=u(x)$ and $\overline{p}_{y_{*}}(\tilde{x}) := p_{y_{*}}(x)$ for the unique $x$ such that $\tilde{x}(x) = \tilde{x}$. We do not transform the $y$-coordinates.
  In these coordinates 
  \begin{align*}
    \tilde{S} &:= \{\tilde{x}; \overline{u}(\tilde{x})\leq \overline{p}_{y_{*}}(\tilde{x}) \}\\
   \text{ and }\quad \tilde{S_0} &:= \{\tilde{x}; 0 \leq \overline{p}_{y_{*}}(\tilde{x})\},
  \end{align*}
  are convex by Lemma \ref{lem:transformations} (ii).

  \textit{Step 4. (The containment condition \eqref{eq:section-containment})} Now we establish \eqref{eq:section-containment}. We only need this estimate in the $\tilde{x}$ coordinates, so it is in that setting we prove it.  We begin with the observation that $\tilde S \subset \tilde S_0$ hence $\tilde S_0^c \subset \tilde S^c$, where $^c$ denotes complementation. This follows from $\overline{u} \geq 0$. Thus to prove \eqref{eq:section-containment} it suffices to prove $\tilde{S} \subset \{\tilde{x}; \tilde{x}^1 \leq Cr\}$ and $\tilde{S_0} \subset \{\tilde{x} ; \tilde{x}^1 \geq -Cr\}$. Let's start with $\tilde{S} \subset \{\tilde{x}; \tilde{x}^1 \leq Cr\}$. Note that $\tilde{x}(re_1)$ is in $\partial S$ since $\overline{u}(\tilde{x}(re_1)) = \overline{p}_{y_{*}}(\tilde{x}(re_1))$. Moreover the gradient to the function defining the section, that is
  \[ D_{\tilde{x}}(\overline{u}-\overline{p}_{y_{*}})(\tilde{x}(re_1)),\]
  will be a normal to a supporting hyperplane of $\tilde{S}$. By the chain rule
  \begin{equation}
    \label{eq:2}
    D_{\tilde{x}}(\overline{u}-\overline{p}_{y_{*}})(\tilde{x}(re_1)) = [b_{xy}(re_1,y_{*})]^{-1}D_{x}(u-p_{y_{*}})(re_1).
  \end{equation}

  By the expansion \eqref{eq:expansion} we have that, for suitable continuous $g_{i}$, 
  \[  b_{xy}(re_1,y_{*}) = I + rg_{i}(re_1,y_{*})(y_{*})^{i} = I+O(r),\]
  here $O(r)$ is a matrix with each component $O(r)$ as $r \rightarrow 0$. That is,
  \[ |b_{x^i,y^j}(re_1,y_{*}) - \delta_{ij}| \leq C r,\]
  for $C$ depending only on $b$. 
  The same is true for the inverse; $[b_{xy}(re_1,y_{*})]^{-1} = I + O(r)$. 
  Thus by employing also \eqref{eq:y1/2-def}, we see \eqref{eq:2} becomes
  \[ D_{\tilde{x}}(\overline{u}-\overline{p})(\tilde{x}(re_1)) = \left(\kappa -
      \frac{h}{2r}\right)e_1 + O(r)\left(\kappa - \frac{h}{2r}\right)e_1.\]
  The same argument used after \eqref{eq:angle-argument} implies that $D_{\tilde{x}}(\overline{u}-\overline{p}_{y_{*}})(\tilde{x}(re_1))$, which is the normal to the supporting hyperplane of $\tilde{S}$ at $\tilde{x}(re_1)$, makes
  angle $\overline{\theta}$ with the $e_1$ axis for
  $\sin \overline{\theta} \leq Cr$, $C$ depending only on $b$. Now because the supporting hyperplane passes through $\tilde{x}(re_1) = re_1+O(r^2)$ and makes angle
  $\sin \overline{\theta}$ with the plane $\tilde{x}^1 = 0$ we have 
  \[ \tilde{S} \subset \{\tilde{x} ; \tilde{x}^1 \leq r+O(r^2)+C(\sin \overline{\theta}) \text{diam}(\tilde{S}) \leq
    Cr \}.\] See Figure \ref{fig:geometric-arguments}(b). A similar argument will be used for the estimate
  \[ \tilde{S_0} \subset \{\tilde{x} ; \tilde{x}^1 \geq -Cr\}.\] First note, using \eqref{eq:expansion} and \eqref{eq:y-expansion}, we obtain
  \begin{align}
   \nonumber    p_{y_{*}}(0) &\geq -y_{*} \cdot (re_1) - Cr^2|y_{*}|^2 + h\\
\nonumber & \geq \left(\frac{-h}{2r}e_1\right)\cdot(re_1) - C|r|^2|y_{*}|^2 + h \\
    \label{eq:origin-est}    &\geq \frac{-h}{2} - Ch^2 +h> \frac{h}{4}.
  \end{align}
In the final line we've used $|y_{*}| \leq C\frac{h}{r}$ and that $Ch^2 \leq h/4$ provided $h$ is sufficiently small (which is assured by $r$ sufficiently small and the locally-Lipschitz property of $u$). 
  On the other hand, similar reasoning yields
  \begin{align*}
    p_{y_{*}}(-2re_1)  &= b(-2re_1,y_{*}) - b(re_1,y_{*})+h\\
                         &\leq  (-2re_1)\cdot y_{*} - (re_1)\cdot y_{*} + h + C|r|^2|y_{*}|^2\\
                         &\leq (-2r)\left(\frac{h}{2r}\right) - r \left(\frac{h}{2r}\right) + h + C|r|^2|y_{*}|^2 < -\frac{h}{4}
  \end{align*}
 Thus in the $x$ coordinates $S_0$ has a boundary point $-te_1$ for some $t \in (-2r,0)$ and subsequently in the $\tilde{x}$ coordinates $\tilde{S_0}$ has boundary point $\tilde{x}(-te_1) = -te_1+O(t^2)$. 
    Similar techniques to above yield that at this boundary point, $\tilde{x}(-te_1)$, the outer normal $D_{\tilde{x}}\overline{p}$ makes angle  $\sin\underline{\theta} \leq Cr$ with the negative $e_1$ axis. The same tangent hyperplane argument as above, this time applied to the convex set $S_0$, implies $\tilde{S_0} \subset \{\tilde{x}; \tilde{x}^1 \geq -Cr \}$. This completes
  the proof of \eqref{eq:section-containment}, the containment condition.

  \textit{Step 5. (Proof of \eqref{eq:energy-comparison})} Now we start the proof of \eqref{eq:energy-comparison}. Our argument relies on \eqref{eq:section-containment}. Thus we prove \eqref{eq:energy-comparison} after the further change of variables \eqref{eq:y1/2-transform}. For brevity we retain the $x,b,p_{y_{*}}$ notation (no tilde).  The estimate \eqref{eq:energy-comparison} will hold for the original coordinates and functions, with the constants $C_1,C_2$ modified by dependence on the coordinate transform, that is, on $b$. 

Now, the uniform $b^*$-convexity condition on $c$ noted in Remark \ref{rem:unif-conv} implies
  \begin{align}
    \nonumber \int_{S} \big[c(Yu(x)) - &b(x,Yu(x)) -c(y_{*}) + b(x,y_{*})\big] \ f(x) \ dx   \\
    \nonumber                      &= \int_{S}\big[c\big(Y(x,Du(x))\big) - b\big(x,Y(x,Du(x))\big) \\
    \nonumber    &\quad\quad-c\big(Y(x,Dp_{y_{*}}(x))\big) + b\big(x,Y(x,Dp_{y_{*}}(x))\big)\big]f(x) \ dx \\
    \label{eq:integrand-ineq}    &\geq \int_{S} \big[\epsilon|Dp_{y_{*}}(x) - Du(x)|^2\\
    \nonumber    &\quad\quad+ D_p[c(y_{*}) - b(x,y_{*})]\cdot\big(Du(x) - Dp_{y_{*}}(x)\big) \big]f(x)\ dx ,
  \end{align}
  where $\epsilon$ derives from the uniform $b^*$-convexity of $c$ and derivatives with respect to $p$ are meant in the sense
  \[ D_{p}[c(y_{*}) - b(x,y_{*})] = D_p[c(Y(x,p))-b(x,Y(x,p))]|_{p = b_x(x,y_{*})}.\] 
  We compute estimates for both terms in \eqref{eq:integrand-ineq}. Using the divergence theorem
  \begin{align}
    \nonumber    \int_{S}D_p&[c(y_{*}) - b(x,y_{*})]\cdot\big(Du(x) - Dp_{y_{*}}(x)\big) f(x)\ dx\\
    \nonumber            &= \int_{S}\nabla_x \cdot \left((u-p_{y_{*}})f(x)D_p[c(y_{*}) - b(x,y_{*})]\right)\\
    \nonumber            &\quad\quad- (u - p_{y_{*}})\nabla_x \cdot\left[f(x)D_p\big(c(y_{*}) - b(x,y_{*})\big)\right]  \ dx \\
    \label{eq:int-est1}    & \geq \int_{\partial S}f(x)(u - p_{y_{*}})D_p[c(y_{*}) - b(x,y_{*})] \cdot \hat{n} \ ds - C|h||S|.
  \end{align}
  The final inequality uses \eqref{eq:height-comparison} to control $p_{y_{*}}-u$ and $f \in  C^{0,1}(\X)$. To compute the boundary integral decompose $\partial S$ as
  $\partial S = (\partial S \cap \X) \cup (\partial S \cap \partial \X)$. Then 
  $p_{y_{*}}-u$ is $0$ on $\partial S \cap \X$ and bounded by $h$ on
  $(\partial S \cap \partial \X)$ so \eqref{eq:int-est1} becomes
  \begin{align}
    \nonumber    \int_{S}D_p[&c(y_{*}) - b(x,y_{*})]\cdot\big(Du(x) - p_{y_{*}}(x)\big)(x) \ dx\\
    \label{eq:int-est2}    &\geq -Ch|\partial S \cap \partial \X| - C h|S| \geq -Ch|S|.
  \end{align}
  Here the final inequality is by the estimate $|\partial S \cap \partial \X| \leq C|S|$ relating the area of $\partial S \cap \partial \X$ to the volume of $S$, proved by Carlier and Lachand-Robert \cite{MR1811127} and Chen \cite[pg. 47]{MR4531938}. Chen's proof is as follows
  \[ |S| = \frac{1}{n}\int_{S} \text{div}(x)  \ dx \geq \frac{1}{n}\int_{\partial S \cap \partial X} x \cdot \hat{n} \ dS \geq C| \partial S \cap \partial \X|. \]
We've used that by convex geometry for $x \in \partial S \cap \partial X$, $x \cdot \hat{n}$ is the distance from the origin to the support plane of $S$ with normal $\hat{n}$. Thus the integrand $x \cdot \hat{n}$ is bounded below by $\text{dist}(\X',  \partial \X)$. 

  To estimate the remaining term in \eqref{eq:integrand-ineq}, that is the term
\[ \int_{S} \epsilon|Dp_{y_{*}}(x) - Du(x)|^2 f(x) \ dx, \]
we follow Caffarelli and Lions's argument. For $x = (x^1,x')$, let $P(x) := (0,x')$  be the projection onto $\{x; x^1 = 0\}$.   We denote by $S/K$ the dilation of $S$ by a factor $1/K$ with respect to the origin where $K>1$ will be chosen. 

For each $(0,x') \in P(S/K)$ the set $(P^{-1}(x') \cap S ) \setminus (S/K)$ is two disjoint line segments. We let $l_{x'}$ be the line segment with greater $x_1$ component and write $l_{x'} = [a_{x'},b_{x'}] \times \{x'\}$ where $b_{x'} > a_{x'}$.

Choose $K = \max\{2\text{diam}(\X)/d,1\}$  where $d = \text{dist}(\X',\partial \X)$. Thus, because $B_{d}(0) \subset \X$,  $P(\X/K)+(d/2)e_1 \subset \text{int}(\X)$. At this point we take $r$ smaller depending on $b,d$ to ensure (via the slab containment \eqref{eq:section-containment})  $S \subset \{x ; |x_1| \leq d/2\}$.
Then for $(0,x') \in P(S/K)$ we see  slab containment implies the point $(b_{x'},x')$ satisfies $b_{x'} \leq d/2$ so $(b_{x'},x') \in \partial S \cap \text{int}X$.

We claim the following 
   \begin{align}
  \label{eq:a-cond}   u((a_{x'},x')) - p_{y_{*}}((a_{x'},x)) &\leq -\frac{K-1}{K}h/4,\\
  \label{eq:b-cond}   u((b_{x'},x')) - p_{y_{*}}((b_{x'},x)) &= 0,\\
  \label{eq:d-cond}   d_{x'}:=b_{x'} - a_{x'} &\leq Cr.
   \end{align}
   Here \eqref{eq:a-cond} is by convexity of $u-p_{y_{*}}$ along a line segment joining the origin, where $u-p_{y_{*}} \leq -h/4$, to $(Ka_{x'},Kx') \in \partial S$, where $u-p_{y_{*}}\leq0$. Then \eqref{eq:b-cond} is by $u = p_{y_{*}}$ on $\partial S \cap \text{int}(\X)$ and \eqref{eq:d-cond} is by the containment condition \eqref{eq:section-containment}.

     Thus, by an application of Jensen's inequality we have
  \begin{align}
    \nonumber    \int_{a_{x'}}^{b_{x'}}&[D_{x^1}p_{y_{*}}((t,x')) - D_{x^1}u((t,x'))]^2 dt\\
    \nonumber    &\geq \frac{1}{d_{x'}}\left( \int_{a_{x'}}^{b_{x'}} D_{x^1}p_{y_{*}}((t,x')) - D_{x^1}u((t,x')) dt\right)^2 \\
                                        &\geq \frac{1}{d_{x'}}\left(\frac{K-1}{K}\frac{h}{4}\right)^2 \geq Ch^2/r.\label{eq:4}
  \end{align}
  To conclude we integrate along all lines $l_{x'}$  for  $x' \in P(S/K)$. Indeed
  \begin{align}
    \nonumber   \int_{S} |Dp_{y_{*}} - Du|^2 \ dx &\geq \int_{P(S/K)}\int_{a_{x'}}^{b_{x'}} |D_{x^1}p_{y_{*}} - D_{x^1}u|^2 \ d t  \ d x'\\
    \nonumber   &\geq \int_{P(S/K)} Ch^2/r \ d x'\\
    \nonumber   &= C\frac{h^2}{r^2} (r |P(S/K)|)\\
    \label{eq:square-est}   &\geq C \frac{h^2}{r^2}|S|.
  \end{align}
  Here the final inequality uses the convexity of $S$ and the containment condition \eqref{eq:section-containment}. 
   Substituting 
  \eqref{eq:int-est2} and \eqref{eq:square-est} into \eqref{eq:integrand-ineq} (recall $f \geq \lambda$) yields the energy estimate  \eqref{eq:energy-comparison}, thereby completing the proof of Lemma \ref{lem:main}.
\end{proof}

\appendix

\section{Proof of Lemma \ref{lem:transformations}}
\label{sec:proof-lemma-reflem}
Here we prove Lemma \ref{lem:transformations}. Whilst results like these appear in a number of works, the techniques we employ in the proof are taken from \cite{MR2654086,MR3067826,trudiner-wang-preprint,MR3419751,MR3121636}.

For parts (i) and (ii) we follow \cite{MR3121636,trudiner-wang-preprint}. Set
\[ h(\tilde{x}) := u(x) - b(x,y_0),\]
for $x$ such that $\tilde{x}(x) = x$. We compute an expression for the Hessian of $h$.  By direct calculation
\[D_{\tilde{x}^i\tilde{x}^j} h = [u_{x^\alpha x^\beta}(x)-b_{x^\alpha x^\beta}(x,y_0)]\frac{\partial x^\alpha}{\partial \tilde{x}^i}\frac{\partial x^\beta}{\partial \tilde{x}^j} + [u_\alpha(x) - b_{\alpha}(x,y_0)]\frac{\partial^2 x^\alpha}{\partial \tilde{x}^i\partial \tilde{x}^j}.\]
The $b$-convexity of $u$ implies
\begin{equation}
  \label{eq:b-convexity-ineq}
  u_{x^\alpha x^\beta}(x) \geq b_{x^\alpha x^\beta}(x,Y(x,Du)) + \epsilon \delta_{ab},
\end{equation}
where $\epsilon>0$ if $u$ is uniformly $b$-convex and $\epsilon=0$ if $u$ is merely $b$-convex. Thus
\begin{align}
\label{eq:second-derivs}  D_{\tilde{x}^i\tilde{x}^j} h &\geq [b_{x^\alpha x^\beta}(x,Y(x,p_1))-b_{x^\alpha x^\beta}(x,Y(x,p_0))]\frac{\partial x^\alpha}{\partial \tilde{x}^i}\frac{\partial x^\beta}{\partial \tilde{x}^j} \\
  &\quad+ [p_1 - p_0]_k\frac{\partial^2 x^k}{\partial \tilde{x}^i\partial \tilde{x}^j}+\epsilon\sum_{\alpha}\frac{\partial x^\alpha}{\partial \tilde{x}^i}\frac{\partial x^\alpha}{\partial \tilde{x}^j},\nonumber
\end{align}
where we've set $p_1 = Du(x), p_0 = b_{x}(x,y_0)$. A direct calculation implies 
\[ \frac{\partial^2 x^k}{\partial \tilde{x}^i\partial \tilde{x}^j} = -D_{p_k}b_{x^\alpha x^\beta}(x,Y(x,p_0)) \frac{\partial x^\alpha}{\partial \tilde{x}^i}\frac{\partial x^\beta}{\partial \tilde{x}^j}.\]
Thus \eqref{eq:second-derivs} may be rewritten as
\begin{align*}
  D_{\tilde{x}^i\tilde{x}^j}h   &\geq [b_{x^\alpha x^\beta}(x,Y(x,p_1))-b_{x^\alpha x^\beta}(x,Y(x,p_0)) \\
  &\quad- (p_1-p_0)_kD_{p_k}b_{x^\alpha x^\beta}(x,Y(x,p_0))] \frac{\partial x^\alpha}{\partial \tilde{x}^i}\frac{\partial x^\beta}{\partial \tilde{x}^j} +  \epsilon\sum_{\alpha}\frac{\partial x^\alpha}{\partial \tilde{x}^i}\frac{\partial x^\alpha}{\partial \tilde{x}^j}.
\end{align*}
Using a Taylor series for the function $p \mapsto b_{x^\alpha x^\beta}(x,Y(x,p))$ we obtain
\[  D_{\tilde{x}^i\tilde{x}^j}h \geq D_{p_kp_l}b_{x^\alpha x^\beta}(x,Y(x,p_\tau))\frac{\partial x^\alpha}{\partial \tilde{x}^i}\frac{\partial x^\beta}{\partial \tilde{x}^j}(p_1-p_0)_k(p_1-p_0)_l+\epsilon\sum_{\alpha}\frac{\partial x^\alpha}{\partial \tilde{x}^i}\frac{\partial x^\alpha}{\partial \tilde{x}^j}.\]
Now we test the convexity in an arbitrary direction. Fix $\tilde{x}_0,\tilde{x}_1 \in \tilde{\X}$ and set $\tilde{h}(t) =h(t\tilde{x}_1+(1-t)\tilde{x}_0)$. The above expression and B3 implies
\begin{align}
  \label{eq:conv-est}
  \tilde{h}''(t) \geq \epsilon\sum_{\alpha}\frac{\partial x^\alpha}{\partial \tilde{x}^i}\frac{\partial x^\alpha}{\partial \tilde{x}^j} (\tilde{x}_1-\tilde{x}_0)_i(\tilde{x}_1-\tilde{x}_0)_j.
\end{align}
To estimate this from below we compute $\frac{\partial x}{\partial \tilde{x}} = [b_{x,y}(x,y_0)]^{-1}$ and by condition B1,  $|\det [b_{x,y}(x,y_0)]^{-1}| > c_0$ for some positive $c_0$. Thus there is $c_1$ such that for all unit vectors $\xi$ we have $|[b_{x,y}(x,y_0)]^{-1}\xi| \geq c_1$ (indeed, $[b_{x,y}(x,y_0)]^{-1}\xi \neq 0$ and compactness implies the lower bound). Note $\frac{\partial x^\alpha}{\partial \tilde{x}^i} (\tilde{x}_1-\tilde{x}_0)_i$ is the $\alpha^{th}$ component of the vector $[b_{x,y}(x,y_0)]^{-1}(\tilde{x}_1-\tilde{x}_0)$. Subsequently \eqref{eq:conv-est} becomes
\[ \tilde{h}''(t) \geq c_1^2 \epsilon |\tilde{x}_1-\tilde{x}_0|^2.\]
This implies the corresponding uniform convexity of $\tilde{u}$ when $\epsilon > 0$ and convexity when $\epsilon = 0$.

Point (ii) of Lemma \ref{lem:transformations} follows immediately from the convexity of $u$ in the given coordinates. 

For result (iii) in Lemma \ref{lem:transformations} we assume $b_{ij}(x_0,y_0) = \delta_{ij}$ which can always be satisfied after an initial affine transformation and $(x_0,y_0)=(0,0)$. Then by a second-order Taylor series for $\tilde{b}$ we have
\begin{equation}
  \label{eq:expansion-1}
  \tilde{b}(\tilde{x},\tilde{y}) = \tilde{b}(0,\tilde{y})+\tilde{b}_{\tilde{x}^i}(0,\tilde{y}) \tilde{x}^i+\frac{1}{2}\tilde{b}_{\tilde{x}^i\tilde{x}^j}(\tilde{x}^t,\tilde{y})\tilde{x}^i\tilde{x}^j.
\end{equation}
From \eqref{eq:tilde-b-def}, $\tilde{b}(0,y) = 0$ is clear. In addition
\begin{equation}
  \label{eq:expansion-2}
  \tilde{b}_{\tilde{x}^i}(0,\tilde{y}) = [b_{x^k}(x_0,y) - b_{x^k}(x_0,y_0)]\left(\frac{\partial x^k}{\partial \tilde{x}^i}\big\vert_{\tilde{x}=0}\right) = \tilde{y}^i.
\end{equation}

Another Taylor series, implies
\begin{align}
\nonumber  \tilde{b}_{\tilde{x}^i\tilde{x}^j}(\tilde{x}_t,\tilde{y}) &=   \tilde{b}_{\tilde{x}^i\tilde{x}^j}(\tilde{x}_t,0) +   \tilde{b}_{\tilde{x}^i\tilde{x}^j\tilde{y}^k}(\tilde{x}_t,0)\tilde{y_k} + \frac{1}{2}\tilde{b}_{\tilde{x}^i\tilde{x}^j\tilde{y}^k\tilde{y}^l}\tilde{y}^k\tilde{y}^l\\
 \label{eq:expansion-3} &= \frac{1}{2}\tilde{b}_{\tilde{x}^i\tilde{x}^j\tilde{y}^k\tilde{y}^l}\tilde{y}^k\tilde{y}^l.
\end{align}
Here we've used that $\tilde{b}({\tilde{x},0})= 0$ and, similarly to \eqref{eq:expansion-2}, $\tilde{b}_{\tilde{y}^k}(\tilde{x},0) = \tilde{x}^k$, whereby the third derivative term vanishes. Equations \eqref{eq:expansion-2} and \eqref{eq:expansion-3} into \eqref{eq:expansion-1} yields Lemma \ref{lem:transformations} (iii).

\bibliographystyle{plain}
\bibliography{bibliography} 
\end{document}